\documentclass[12pt,reqno,a4paper]{amsart}
\usepackage{euscript}

\newtheorem{theorem}{Theorem}

\newtheorem{proposition}[theorem]{Proposition}
\newtheorem{lemma}{Lemma}
\newtheorem{remark}{Remark}
\newtheorem{example}{Example}

\renewcommand{\epsilon}{\varepsilon}
\DeclareMathOperator{\Ima}{Im}
\DeclareMathOperator{\Ker}{Ker}

\def\Id{\text{\rm Id}}

\def\N{\mathbb{N}}

\def\R{\mathbb{R}}
\begin{document}

\title[Admissibility and polynomial dichotomies]{Admissibility and polynomial dichotomies for evolution families}
\begin{abstract}
For an arbitrary evolution family, we consider the notion of a
polynomial  dichotomy with respect to a family of norms and characterize it 
in terms of the admissibility property, that is, the existence
of a unique bounded solution for each bounded perturbation. In particular, by
considering a family of Lyapunov norms, we recover the notion of a (strong) nonuniform
polynomial dichotomy. As a nontrivial application of the characterization, we establish
the robustness of the notion of a strong  nonuniform  polynomial dichotomy 
 under sufficiently small linear perturbations. 
\end{abstract}
\begin{thanks}
{ D.D. was supported by the  Croatian Science Foundation under the project IP-2014-09-2285}
\end{thanks}

\author{Davor Dragi\v cevi\'c}
\address{Department of Mathematics, University of Rijeka, Croatia}
\email{ddragicevic@math.uniri.hr}
\maketitle

\section{Introduction}
For a nonautonomous linear equation 
\[
x'=A(t)x \quad t\ge 1,
\]
on a Banach space defined by a continuous function $A(t)$  and, more generally,
for an arbitrary evolution family $T (t, s)$ on a Banach space, we consider the notion of a polynomial dichotomy  with respect to a family of norms.  Similar to the classical notion of an exponential dichotomy (essentially introduced by Perron~\cite{Pe}), the notion of a polynomial dichotomy 
requries  for the phase space to split into two complementary directions, called the stable and the unstable direction, such that the dynamics exhibits contraction along the stable direction and expansion along the unstable direction. However, in the case of polynomial dichotomies, the rates of contractions and expansion are polynomial. 
The main objective for considering the notion of a polynomial dichotomy \emph{with respect to a family of norms} is that it includes as  a particular case:
\begin{itemize}
\item the notions of a uniform polynomial stability and expansivity studied by Hai~\cite{Hai};
\item the notion of a nonuniform polynomial dichotomy introduced independently by Barreira and Valls~\cite{BV1} and Bento and Silva~\cite{BS1, BS2}.
\end{itemize}
We emphasize that the more general notion  of dichotomy (associated to a certain growth rate)  was introduced and studied in~\cite{BV}). These developments can be seen as a contribution to the line of the research initiated by Muldowney~\cite{M} and Naulin and Pinto~\cite{NP, NP1}, who were the first  to consider (uniform) dichotomies 
with non-exponential rates of contraction and expansion.

The main objective of the present paper is to establish results analogous to those obtained in~\cite{BDV3} but for polynomial dichotomies. More precisely, we show that the notion of a polynomial dichotomy with respect to a family of norms can  be (under  suitable additional requirement that the evolution exhibits polynomial bounded growth) completely characterized in terms of the appropriate admissibility property, that is, in terms of the existence of a unique bounded solution
$x$ of the equation
\begin{equation}
x(t)=T(t,\tau)x(\tau)+\int_\tau^t \frac 1 s T(t,s)y(s)\, ds, \quad t\ge \tau \ge 1, 
\end{equation}
for each bounded perturbation $y$. We refer to Section~\ref{MR} for details. 

We emphasize that the study of the admissibility in relation to \emph{exponential} dichotomies goes back to the landmark works of Perron~\cite{Pe} and Li~\cite{tali} but was first systemically studied in seminal works of Massera and Sch\"affer~\cite{MS1, MS} (see also Coppel~\cite{Co}). The  first results that deal  with the case of  infinite-dimensional dynamics are due to  Dalec$'$ki\u{\i} and  Kre\u\i{n}~\cite{DK} in the case of continuous time and by Henry~\cite{He} for noninvertible dynamics with discrete time. The case of exponential dichotomies on the half-line was first considered (in the infinite-dimensional case) in~\cite{Sx2}. For more recent results, we refer to the works of Huy~\cite{HH},  Latushkin,  Randolph and  Schnaubelt~\cite{LRS},  Preda, Pogan and Preda~\cite{PPP,PPP1} as well as Sasu and Sasu~\cite{SS1,SS2, SS3}. For results dealing with various flavours of nonuniform behaviour, we refer to~\cite{BDV1, BDV3, LP, MSS1, MP, SBS, ZZ, ZLZ2} and references therein. Finally, for a detailed exposition and additional references, we recommend~\cite{BDV2}.

To the best of our knowledge, the first contribution to the study of the relationship between admissibility and polynomial asymptotic behaviour is due to Hai~\cite{Hai}.  However, in~\cite{Hai} the author deals only with uniform polynomial stability and expansivity.  In particular, the case of dichotomies is not considered. 
Moreover, our admissibility spaces are different from those used in~\cite{Hai}.  More recently, the author has developed results similar to those in the present paper for discrete-time dynamics~\cite{DD}. Although the approach  in the present paper is similar to that in~\cite{DD}, we emphasize that it requires nontrivial changes.

The paper is organized as follows. In Section~\ref{P} we introduce the notion of a polynomial dichotomy with respect to a family  of norms. Then, in Section~\ref{MR} we obtain a complete characterization of this notion in terms of the appropriate admissibility property. Finally, in Sections~\ref{NPD} and~\ref{R} we apply our results to the study of nonuniform polynomial dichotomies.

\section{Preliminaries}\label{P}
Let $X=(X, \lVert \cdot \rVert)$ be a Banach space and let $B(X)$ denote the space of all bounded linear operators on $X$.   A family $T(t, \tau)$, $t\ge \tau \ge 1$, of bounded linear operators is said to be an \emph{evolution family} if
\begin{enumerate}
\item $T(t,t)=\Id$ for $t\ge 1$;
\item $T(t,s)T(s,\tau)=T(t, \tau)$ for $t\ge s\ge \tau \ge 1$;
\item given $\tau \ge 1$ and $x\in X$, the map $s\mapsto T(s,\tau)x$ is continuous on $[\tau, \infty)$.
\end{enumerate}
We also consider a family of norms  $\lVert \cdot \rVert_t$ for $t\ge 1$  on $X$ such that 
\begin{itemize}
\item[(i)] there exist $C>0$ and $\epsilon \ge 0$ such that $\lVert x\rVert \le \lVert x\rVert_\tau \le C\tau^\epsilon \lVert x\rVert$ for every $x\in X$ and $\tau \ge 1$;
\item[(ii)] the map $t\mapsto \lVert x\rVert_t$ is measurable for each $x\in X$.
\end{itemize}

We say that  an evolution family $T(t,\tau)$ admits a \emph{polynomial dichotomy} with respect to the family of norms $\lVert \cdot \rVert_t$ 
if:
\begin{enumerate}
\item there exist projections $P(t)$ for $t\ge 1$ satisfying
\begin{equation}\label{pro}
P(t)T(t, \tau)=T(t,\tau)P(\tau) \quad \text{for $t\ge \tau \ge 1$}
\end{equation}
such that the map  $T(t,\tau)\rvert_{\Ker P(\tau)} \colon \Ker P(\tau)  \to \Ker P(t)$ is invertible for all $t\ge \tau \ge 1$;
\item there exist $\lambda, D>0$ such that for every $x\in X$ and $t,\tau \ge 1$, we have 
\begin{equation}\label{d1}
\lVert T(t,\tau)P(\tau) x\rVert_t \le D(t/\tau)^{-\lambda} \lVert x\rVert_\tau \quad \text{for $t\ge \tau$}
\end{equation}
and 
\begin{equation}\label{d2}
\lVert T(t,\tau)Q(\tau) x \rVert_t \le  D(\tau /t)^{-\lambda}\lVert x\rVert_\tau \quad \text{for $t\le \tau$,}
\end{equation}
where $Q(t)=\Id-P(t)$ and
\[
T(t,\tau)=(T(\tau, t)\rvert_{\Ker P(t)})^{-1} \colon \Ker P(\tau) \to \Ker P(t)
\]
for $t<\tau$.
\end{enumerate}

We also introduce function spaces that will play a major role in our arguments. 
Let $Y$ be the set of all continuous functions $x\colon [1, \infty) \to X$  such that
\[
\lVert x\rVert_\infty:=\sup_{t\ge 1 }\lVert x(t)\rVert_t <\infty.
\]
It is easy to prove that $(Y, \lVert \cdot \rVert_\infty)$ is a Banach space. Furthermore, for a given closed subspace $Z\subset X$, let $Y_Z$ be the set of all $x \in Y$ such that $x(1)\in Z$. It is easy to verify that $Y_Z$ is a closed subspace of $Y$. 
We will write $Y_0$ instead of $Y_{\{0\}}$.

Furthermore, we  also consider the set $Y_1$ of all locally integrable functions $x\colon [1, \infty) \to X$ such that 
\[
\lVert x\rVert_L:=\sup_{t\ge 1}\int_t^{t+1}\lVert x(s)\rVert_s\, ds<\infty.
\]
Then one can easily  prove that $(Y_1, \lVert \cdot \rVert_L)$ is a Banach space.

\section{Main results}\label{MR}
The following  is our first main result. 
\begin{theorem}\label{t2}
Assume that the evolution family $T(t,\tau)$ admits a polynomial dichotomy with respect to the family of norms $\lVert \cdot \rVert_t$. Then, for each $y\in Y_1$, there exists a unique $x\in Y_Z$, with $Z=\Ima Q(1)$, such that
\begin{equation}\label{503a}
x(t)=T(t,\tau)x(\tau)+\int_\tau^t \frac{1}{s} T(t,s)y(s)\, ds.
\end{equation}
\end{theorem}

\begin{proof}
  Observe that without any loss of generality we can assume that~\eqref{d1} and~\eqref{d2} hold with $\lambda \in (0, 1)$.
 Choose $y\in Y_1$ and extend it to the function $y\colon \R \to X$  by $y(t)=0$ for $t<1$. Let us prove that there exists $x\in Y_Z$ such that~\eqref{503a} holds.  For each $t\ge 1$, let
\[
x_1(t)=\int_t^\infty \frac{1}{\tau} T(t,\tau)Q(\tau)y(\tau)\, d\tau 
\]
and
\[
x_2(t)=\int_1^t \frac{1}{\tau}T(t, \tau)P(\tau)y(\tau)\, d\tau.
\]
By applying~\eqref{d2}, we obtain that 
\[
\begin{split}
\lVert x_1(t)\rVert_t &\le \int_t^\infty \frac{1}{\tau} \lVert  T(t,\tau)Q(\tau)y(\tau)\rVert_t \, d\tau \\
&\le D\int_t^\infty \frac{1}{\tau}(\tau/t)^{-\lambda} \lVert y(\tau)\rVert_\tau\, d\tau \\
&=D\sum_{m=0}^\infty \int_{t+m}^{t+m+1}\frac{1}{\tau}(\tau/t)^{-\lambda} \lVert y(\tau)\rVert_\tau\, d\tau \\
&\le D \sum_{m=0}^\infty \frac{t^\lambda}{(t+m)^{1+\lambda}}\int_{t+m}^{t+m+1} \lVert y(\tau)\rVert_\tau\, d\tau \\
&\le Dt^\lambda \bigg{(}\frac{1}{t^{1+\lambda}}+\int_0^\infty \frac{ds}{(t+s)^{1+\lambda}} \bigg{)}\lVert y\rVert_L \\
&\le D(1+\lambda^{-1})\lVert y\rVert_L,
\end{split}
\]
for $t\ge 1$.  Similarly, it follows from~\eqref{d1} that 
\[
\begin{split}
\lVert x_2(t)\rVert_t &\le \int_1^t \frac{1}{\tau}\lVert T(t,\tau)P(\tau)y(\tau)\rVert_t\, d\tau \\
&\le D\sum_{m=0}^{\lceil t\rceil -3}\int_{t-m-1}^{t-m}\frac{1}{\tau} (t/\tau)^{-\lambda} \lVert  y(\tau)\rVert_\tau\, d\tau +D \int_1^{t-\lceil t\rceil +2} \frac{1}{\tau} (t/\tau)^{-\lambda} \lVert  y(\tau)\rVert_\tau\, d\tau\\
&\le Dt^{-\lambda}\sum_{m=0}^{\lceil t\rceil -3}\frac{1}{(t-m-1)^{1-\lambda}}\int_{t-m-1}^{t-m}\lVert  y(\tau)\rVert_\tau\, d\tau  +Dt^{-\lambda}\lVert y\rVert_L \\
&\le  Dt^{-\lambda}\int_0^{\lceil t\rceil -2}\frac{ds}{(t-s-1)^{1-\lambda}}\lVert y\rVert_L+D\lVert y\rVert_L \\
&\le D(1+\lambda^{-1})\lVert y\rVert_L,
\end{split}
\]
for $t> 2$. Observe also that the above inequality trivially holds for $t\in [1, 2]$. Set $x(t)=x_2(t)-x_1(t)$, $t\ge 1$. Obviously,  it follows from the estimates above that  $\sup_{t\ge 1} \lVert x(t)\rVert_t <\infty$ and $x$ is continuous.  Moreover, for $t\ge \tau \ge 1$ we have that 
\[
\begin{split}
x(t) &=\int_\tau^t \frac 1 s T(t,s)y(s)\, ds-\int_\tau^t \frac 1 s T(t,s)P(s) y(s)\, ds \\
&\phantom{=}-\int_\tau^t \frac 1 s T(t,s)Q(s) y(s)\, ds+\int_1^t  \frac 1 s T(t,s)P(s)y(s)\, ds \\
&\phantom{=}-\int_t^\infty \frac 1 s   T(t,s)Q(s)y(s)\, ds \\
&=\int_\tau^t \frac 1 s T(t,s)y(s)\, ds+\int_1^\tau \frac 1 s T(t,s)P(s)y(s)\, ds \\
&\phantom{=}-\int_\tau^\infty \frac 1 s   T(t,s)Q(s)y(s)\, ds \\
&=T(t,\tau)x(\tau)+\int_\tau^t \frac 1 s T(t,s)y(s)\, ds,
\end{split}
\]
and thus~\eqref{503a} holds. Moreover, $P(1)x(1)=0$ and thus $x(1)\in Z$. Consequently, $x\in Y_Z$.

Now we establish the uniqueness of x. It suffices to show that if $x(t)=T(t, \tau)x(\tau)$  for $t\ge \tau \ge 1$ with $x\in Y_Z$, then $x(t)=0$ for $t\ge 1$. It follows from~\eqref{d2} that 
\[
\lVert Q(1)x(1)\rVert_1=\lVert T(1,t)Q(t)x(t)\rVert_1 \le Dt^{-\lambda}\lVert x(t)\rVert_t \le Dt^{-\lambda} \lVert x\rVert_\infty
\]
for $t\ge 1$. Letting $t\to \infty$, we obtain that $x(1)=Q(1)x(1)=0$, and thus $x=0$. The proof of the theorem is completed.

\end{proof}

Let us now establish a partial  converse of Theorem~\ref{t2}.
\begin{theorem}\label{t1}
Assume that there exists a closed subspace $Z\subset X$  such that for each $y\in Y_1$ there exists a unique $x\in Y_Z$ satisfying~\eqref{503a}. 
 Furthermore, suppose that there exist $M, a>0$ such that
\begin{equation}\label{pb}
\lVert T(t,s)x\rVert_t \le M (t/s)^a \lVert x\rVert_s \quad \text{for $t\ge s\ge 1$ and $x\in X$.}
\end{equation}
 Then,  $T(t,\tau)$  admits a polynomial dichotomy with respect to the family  of norms $\lVert \cdot \rVert_t$. 
\end{theorem}

\begin{proof}
Let $T_Z$ be the linear operator defined by $T_Z x=y$ on the domain $\mathcal{D}(T_Z)$ formed by all $x\in Y_Z$ for which there exists $y\in Y_1$ such that~\eqref{503a} holds. 
It is easy to verify that $T_Z$ is well-defined.  Indeed, assume that $x\in Y_Z$ and $y_1, y_2\in Y_1$ are such that 
\[
x(t)=T(t,\tau)x(\tau)+\int_\tau^t \frac 1 s T(t,s)y_i(s)\, ds,
\]
for $t\ge \tau \ge 1$ and $i\in \{1, 2\}$. Hence,
\[
\int_\tau^t \frac 1 s T(t,s)(y_1(s)-y_2(s))\, ds=0,
\]
for $t> \tau \ge 1$. Dividing by $t-\tau$ and letting $t-\tau \to 0$, we obtain that
\[
y_1(t)=y_2(t) \quad \text{for a.e. $t\ge 1$.}
\]
We conclude that  $y_1=y_2$ and thus $T_Z$ is well-defined. 
\begin{lemma}\label{9:26}
The operator $T_Z \colon \mathcal D(T_Z) \to Y_1$ is closed. 
\end{lemma}

\begin{proof}[Proof of the lemma]
Let $(x_n)_{n\in \N}$ be the sequence in $\mathcal D(T_Z)$ converging to $x\in Y_Z$ such that $y_n=T_Z x_n$ converges to $y\in Y_1$. Then for $t\ge \tau \ge 1$, we have that 
\[
\begin{split}
x(t)-T(t,\tau)x(\tau) &=\lim_{n\to \infty} (x_n(t)-T(t,\tau)x_n(\tau))\\
&=\lim_{n\to \infty}\int_\tau^t  \frac 1 s T(t,s)y_n(s)\, ds.
\end{split}
\]
On the other hand,  we have 
\[
\begin{split}
\bigg{\lVert}\int_\tau^t  \frac 1 s T(t,s)y_n(s)\, ds-\int_\tau^t  \frac 1 s T(t,s)y(s)\, ds \bigg{\rVert} &\le M\int_\tau^t \lVert y_n(s)-y(s)\rVert\, ds \\
&\le M\int_\tau^t \lVert y_n(s)-y(s)\rVert_s \, ds \\
&\le M(t-\tau+1)\lVert y_n-y\rVert_L,
\end{split}
\]
where $M=\sup \{\lVert T(t,s)\rVert: s\in [\tau, t]\}$ is finite by the Banach-Steinhaus theorem.  Since $y_n \to y$ in $Y_1$, we conclude that
\[
\lim_{n\to \infty} \int_\tau^t  \frac 1 s T(t,s)y_n(s)\, ds=\int_\tau^t \frac 1 s T(t,s)y(s)\, ds,
\]
and therefore~\eqref{503a} holds. We conclude that $x\in \mathcal D(T_Z)$ and $T_Zx=y$. This completes the proof of the lemma. 
\end{proof}
It follows from the assumptions of the theorem that $T_Z$ is bijective. Hence, by Lemma~\ref{9:26} and the Closed Graph Theorem, we conclude that $T_Z$ has a bounded inverse $G_Z\colon Y_1 \to Y_Z$. For $\tau \ge 1$, we define
\[
S(\tau):=\bigg{\{}x\in X: \sup_{t\ge \tau}\lVert T(t, \tau)x\rVert_t <\infty \bigg{ \} } \quad \text{and} \quad U(\tau)=T(\tau, 1)Z.
\]
Observe that $S(\tau)$ and $U(\tau)$ are subspaces of $X$.

\begin{lemma}
For $\tau \ge 1$, we have that
\begin{equation}\label{1016}
X=S(\tau)\oplus U(\tau).
\end{equation}

\end{lemma}

\begin{proof}[Proof of the lemma]
Take $x\in X$ and $\tau \ge 1$ and set $\alpha (\tau):=\ln (1+\tau^{-1})$.
 We define $g\colon [1, \infty)\to X$ by $g(s)=\chi_{[\tau, \tau+1]}(s)T(s, \tau)x$, $s\ge 1$. Clearly, $g\in Y_1$. 
Since $T_Z$ is bijective, there exists $v\in Y_Z$ such that $T_Zv=g$. It follows from~\eqref{503a} that 
$v(t)=T(t,\tau)(v(\tau)+\alpha (\tau)x)$ for $t\ge \tau+1$. Since $v\in Y_\infty$, we have that 
$v(\tau)+\alpha (\tau)x \in S(\tau)$. On the other hand, \eqref{503a}also  implies that 
$v(\tau)=T(\tau, 1)v(1)$. Since $v\in Y_Z$, we have that $v(1)\in Z$ and thus $v(\tau)\in U(\tau)$. Consequently, 
\[
x=\frac{1}{\alpha(\tau)}(v(\tau)+\alpha (\tau)x)-\frac{1}{\alpha(\tau)} v(\tau)\in S(\tau)+U(\tau).
\]

Now take $x\in S(\tau)\cap U(\tau)$ and choose $z\in Z$ such that $x=T(\tau, 1)z$. We define $u\colon [1, \infty) \to X$ by
$u(t)=T(t,1)z$, $t\ge 1$. Obviously $u\in Y_Z$ and $T_Zu=0$.  Since $T_Z$ is bijective, we have that $u=0$ and thus $x=u(\tau)=0$. We conclude that
$S(\tau)\cap U(\tau)=\{0\}$ and therefore~\eqref{1016} holds. 
\end{proof}

Let $P(\tau)\colon X \to S(\tau)$  and $Q(\tau)\colon X\to U(\tau)$ be the projections associated with the decomposition~\eqref{1016}, with $P(\tau)+Q(\tau)=\Id$.  Observe that~\eqref{pro} holds. 

\begin{lemma}
For $t\ge \tau \ge 1$, the map $T(t,\tau)\rvert_{U(\tau)} \colon U(\tau)\to U(t)$ is invertible.

\end{lemma}

\begin{proof}
Take $x\in U(t)$. Then, there exists $z\in X$ such that $x=T(t,1)z$. Since $T(\tau, 1)z\in U(\tau)$ and $x=T(t, \tau)T(\tau, 1)z$, we conclude that $T(t,\tau)\rvert_{U(\tau)}$ is surjective. 

Assume now that $T(t,\tau)x=0$ for some $x\in U(\tau)$. Take $z\in Z$ such that $x=T(\tau, 1)z$. We define $u\colon [1, \infty)\to X$ by $u(s)=T(s,1)z$, $s\ge 1$.
Since $u(s)=0$ for $s\ge t$ and thus $u\in Y_Z$.  Moreover, $T_Zu=0$ and consequently $u=0$ and $x=u(\tau)=0$. This proves that $T(t,\tau)\rvert_{U(\tau)}$ is also injective.

\end{proof}

We now show that all elements of $S(\tau)$ uniformly polynomially  contract under the action of the evolution family $T(t,\tau)$.

\begin{lemma}
There exist $D, \lambda >0$ such that
\begin{equation}\label{201}
\lVert T(t, \tau)x\rVert_t \le D(t/\tau)^{-\lambda}\lVert x\rVert_\tau, \quad \text{for $t\ge \tau \ge 1$ and $x\in S(\tau)$.}
\end{equation}
\end{lemma}

\begin{proof}[Proof of the lemma]
We first claim that there exists $L>0$ such that 
\begin{equation}\label{o}
\lVert T(t, \tau)v\rVert_t \le L\lVert v\rVert_\tau \quad \text{for $t\ge \tau \ge 1$ and $v\in S(\tau)$.}
\end{equation}
Let us first consider the case when $t\ge 2\tau$ and take $v\in X$ such that $T(t, \tau)v\neq 0$.  Consequently, $T(s,\tau)v\neq 0$ for $\tau \le s \le t$. Let us consider $x, y\colon [1, \infty) \to X$  defined by 
\[
y(s)=\begin{cases}
0 & \text{if $1\le s \le \tau$,} \\
\frac{T(s,\tau) v}{\lVert T(s,\tau)v\rVert_s} & \text{if $\tau < s\le  t$,}\\
0 & \text{if $s>t$,}
\end{cases}
\]
and 
\[
x(s)=\begin{cases}
0 & \text{if $1\le s \le \tau$,} \\
\int_\tau^s  \frac{T(s, \tau)v}{r\lVert T(r, \tau)v\rVert_r} \, dr & \text{if $\tau < s \le t$,}\\
\int_\tau^t  \frac{T(s, \tau)v}{r\lVert T(r, \tau)v\rVert_r}\, dr & \text{if $s>t$.}
\end{cases}
\]
Note that $y\in Y_1$. Furthermore, since $v\in S(\tau)$ we have that $x\in Y_Z$. It is  straightforward to verify that $T_Zx=y$. Consequently,
\[
\lVert x\rVert_\infty =\lVert G_Z y\rVert_\infty \le \lVert G_Z\rVert \cdot \lVert y\rVert_L \le \lVert G_Z\rVert.
\]
Therefore, 
\begin{equation}\label{1245}
\lVert G_Z\rVert \ge \lVert x\rVert_\infty \ge \lVert x(t)\rVert_t = \lVert T(t, \tau)v\rVert_t \int_\tau^t \frac{1}{r\lVert T(r, \tau)v\rVert_r}\, dr,
\end{equation}
and thus 
\[
\lVert G_Z \rVert \ge \lVert T(t, \tau)v\rVert_t \int_\tau^{2\tau} \frac{1}{r\lVert T(r, \tau)v\rVert_r}\, dr.
\]
On the other hand, \eqref{pb} implies that
\[
\lVert T(r, \tau)v\rVert_r \le M(r/\tau)^a \lVert v\rVert_\tau \le M2^a\lVert v\rVert_\tau, \quad \text{for $\tau \le r\le 2\tau$.}
\]
Hence, 
\[
\lVert G_Z\rVert \ge \frac{\lVert T(t, \tau)v\rVert_t}{M2^a\lVert v\rVert_\tau} \int_\tau^{2\tau} \frac 1 r\, dr \ge \ln 2   \frac{\lVert T(t, \tau)v\rVert_t}{M2^{a}\lVert v\rVert_\tau},
\]
which yields 
\begin{equation}\label{x1}
\lVert T(t, \tau)v\rVert_t \le \frac{ M2^{a}\lVert G_Z\rVert}{\ln 2} \cdot \lVert v\rVert_\tau.
\end{equation}
Moreover, \eqref{pb} implies that 
\begin{equation}\label{x2}
\lVert T(t, \tau)v\rVert_t  \le M2^a \lVert v\rVert_\tau \quad \text{for $\tau \le t\le 2\tau$ and $v\in X$.}
\end{equation}
By~\eqref{x1} and~\eqref{x2}, we conclude that~\eqref{o} holds with
\[
L:=\max \bigg{\{} M2^a, \frac{M2^{a}\lVert G_Z\rVert }{\ln 2} \bigg{\}}>0. 
\]
We next show that there exists $N_0\in \N$  such that
\begin{equation}\label{1257}
\lVert T(t, \tau)v\rVert_t \le e^{-1} \lVert v\rVert_\tau \quad \text{for $t\ge N_0\tau$ and $v\in X$.}
\end{equation}
We have (using~\eqref{o} and~\eqref{1245}) that 
\[
\begin{split}
\lVert G_Z\rVert &\ge \lVert T(t, \tau)v\rVert_t \int_\tau^t \frac{1}{r\lVert T(r, \tau)v\rVert_r}\, dr \\
&\ge  \lVert T(t, \tau)v\rVert_t \int_\tau^t \frac{1}{rL \lVert v\rVert_\tau}\, dr \\
&\ge  \lVert T(t, \tau)v\rVert_t \int_\tau^{N_0\tau} \frac{1}{rL \lVert v\rVert_\tau}\, dr \\
&\ge \ln N_0\frac{\lVert T(t, \tau)v\rVert_t }{L\lVert v\rVert_\tau}
\end{split}
\]
and therefore
\[
\lVert T(t, \tau)v\rVert_t \le \frac{L\lVert G_Z \rVert}{\ln N_0}\lVert v\rVert_\tau.
\]
Hence, if choose $N_0$ large enough so that
\[
\frac{L\lVert T_Z^{-1}\rVert}{\ln N_0}\le e^{-1},
\]
we conclude that~\eqref{1257} holds. 

Take now arbitrary $t\ge \tau$, $v\in S(\tau)$  and choose largest $l\in \N\cup \{0\}$ such that $N_0^l \le t/\tau$.  It follows from~\eqref{o} and~\eqref{1257} that 
\[
\lVert T(t,\tau)v\rVert_t =\lVert T(t, N_0^l \tau)T(N_0^l \tau, \tau)v\rVert_t \le Le^{-l}\lVert v\rVert_\tau.
\]
Since $t/\tau <N_0^{l+1}$, we have that
\[
l> \frac{\ln t/\tau}{\ln N_0}-1,
\]
and thus
\[
e^{-l}\le e(t/\tau)^{-1/\log N_0}.
\]
Consequently, 
\[
\lVert  T(t, \tau)v\rVert_t  \le Le (t/\tau)^{-1/\log N_0} \lVert v\rVert_\tau, 
\]
and we conclude that~\eqref{201} holds with
\[
D=Le \quad \text{and} \quad \lambda=1/\log N_0.
\]
The proof of the lemma is completed. 
\end{proof}
Next we show that nonzero vectors in $Z(\tau)$ exhibit uniform polynomial expansion under the action of the evolution family $T(t,\tau)$.
\begin{lemma}
There exist $D, \lambda >0$ such that
\begin{equation}\label{202}
\lVert T(t, \tau)v\rVert_t \le  D (\tau/t)^{-\lambda}\lVert v\rVert_\tau, \quad \text{for $t\le \tau$ and $v\in U(\tau)$.}
\end{equation}
\end{lemma}

\begin{proof}[Proof of the lemma]
Take $z\in Z\setminus \{0\}$ and $\tau \ge 1$. We consider $x, y\colon [1, \infty) \to X$ defined by
\[
y(s)=\begin{cases}
-\frac{T(s, 1)z}{\lVert T(s, 1)z\rVert_s} & \text{if $1\le s \le \tau$;}\\
0 & \text{if $s>\tau$,}
\end{cases} 
\]
and 
\[
x(s)=\begin{cases}
\int_s^\tau \frac{T(s, 1)z}{r\lVert T(r, 1)z\rVert_r} & \text{if $1\le s\le \tau$;}\\
0 & \text{if $s>\tau$.}
\end{cases}
\]
Observe that $y\in Y_1$ and $x\in Y_Z$. Furthermore, it is straightforward to verify that $T_Z x=y$. Hence,
\[
\lVert x\rVert_\infty =\lVert G_Z y\rVert_\infty  \le \lVert G_Z\rVert \cdot \lVert y\rVert_L=\lVert G_Z\rVert.
\]
Therefore, for each $1\le s\le \tau$, we have that
\[
\lVert G_Z\rVert \ge \lVert T(s, 1)z\rVert_s \int_s^\tau \frac{1}{r\lVert T(r, 1)z\rVert_r}\, dr.
\]
Letting $\tau\to \infty$, we conclude that 
\begin{equation}\label{re}
\lVert G_Z\rVert \ge \lVert T(s, 1)z\rVert_s \int_s^\infty  \frac{1}{r\lVert T(r, 1)z\rVert_r}\, dr,
\end{equation}
for each $s\ge 1$ and $z\in Z\setminus \{0\}$.
We now claim that there exists $L>0$ such that 
\begin{equation}\label{215}
\lVert T(t, 1)z\rVert_t \ge L\lVert T(\tau, 1)z\rVert_\tau \quad \text{for $t\ge \tau$ and $z\in Z$.}
\end{equation}
Using~\eqref{pb} and~\eqref{re}, we have that 
\[
\begin{split}
\frac{1}{\lVert T(\tau,1)z\rVert_s} &\ge \frac{1}{\lVert G_Z\rVert }\int_\tau^\infty \frac{1}{r\lVert T(r, 1)z\rVert_r} \, dr\\
&\ge \frac{1}{\lVert G_Z\rVert }\int_t^{2t} \frac{1}{r\lVert T(r, 1)z\rVert_r} \, dr \\
&= \frac{1}{\lVert G_Z\rVert }\int_t^{2t} \frac{1}{r\lVert T(r, t)T(t,1) z\rVert_r} \, dr  \\
&\ge   \frac{1}{\lVert G_Z\rVert }\int_t^{2t} \frac{1}{r M(r/t)^a \lVert T(t, 1)z\rVert_t}\, dr \\
&\ge  \frac{1}{M2^a \lVert G_Z\rVert \cdot \lVert T(t, 1)z\rVert_t }\int_{t}^{2t}  \frac 1 r \, dr \\
&=\frac{\ln 2}{M2^a \lVert G_Z\rVert \cdot \lVert T(t, 1)z\rVert_t },
\end{split}
\]
which readily implies that~\eqref{215} holds with 
\[
L=\frac{\ln 2}{M2^{a} \lVert G_Z\rVert}. 
\]
We next claim that there exists $N_0\in \N$ such that 
\begin{equation}\label{239}
\lVert T(t, 1)z\rVert_t \ge e\lVert T(\tau, 1)z\rVert_\tau \quad \text{for $t\ge N_0\tau$ and $z\in Z$.}
\end{equation}
Indeed, it follows from~\eqref{re} and~\eqref{215} that 
\[
\begin{split}
\frac{1}{\lVert T(\tau,1)z\rVert_\tau} &\ge \frac{1}{\lVert G_Z\rVert }\int_\tau^\infty \frac{1}{r\lVert T(r, 1)z\rVert_r}\, dr \\
&\ge \frac{1}{\lVert G_Z\rVert }\int_{\tau}^{N_0\tau}\frac{1}{r\lVert T(r, 1)z\rVert_r}\, dr \\
&\ge \frac{L}{\lVert G_Z\rVert \cdot \lVert T(t, 1)z\rVert_t}\int_{\tau}^{N_0\tau} \frac 1 r\, dr \\
&\ge \frac{L \ln N_0}{\lVert G_Z\rVert \cdot \lVert T(t, 1)z\rVert_t}.
\end{split}
\]
Hence, if we choose $N_0$ such that
\[
\frac{L \ln  N_0 }{\lVert G_Z\rVert}\ge e, 
\]
we have that~\eqref{239} holds. Proceeding as in the proof of the previous lemma, one can easily conclude that there exist $D, \lambda>0$ such that
\begin{equation}\label{m}
\lVert T(t, 1)z\rVert_t \ge \frac 1 D (t/\tau)^\lambda \lVert T(\tau, 1)z\rVert_\tau \quad \text{for $t\ge \tau$ and $z\in Z$.}
\end{equation}
Take now $t\ge \tau$, $v\in U(\tau)$ and choose $z\in Z$ such that $v=T(\tau, 1)z$. Then, \eqref{m} implies  that 
\[
\begin{split}
\lVert T(t, \tau)v\rVert_t &=\lVert T(t, 1)z\rVert_t  \\
&\ge \frac 1 D (t/\tau)^\lambda \lVert T(\tau, 1)z\rVert_\tau\\
&= \frac 1 D (t/\tau)^\lambda \lVert v\rVert_\tau,
\end{split}
\]
which readily implies that~\eqref{202} holds.
\end{proof}
  The final ingredient of the proof is the following lemma. 
\begin{lemma}
We have that
\begin{equation}\label{203}
\sup_{\tau \ge 1}\lVert P(\tau)\rVert <\infty. 
\end{equation}
\end{lemma}

\begin{proof}[Proof of the lemma]
For each $\tau \ge 1$, let
\[
\gamma (\tau):=\inf \{\lVert v^s+v^u\rVert_\tau: \lVert v^s\rVert_\tau=\lVert v^u\rVert_\tau=1, \ v^s\in S(\tau), \ v^u \in U(\tau)\}.
\]
Then (see~\cite[Lemma 4.2]{Sx2}), 
\begin{equation}\label{Pn}
\lVert P(\tau)\rVert \le \frac{2}{\gamma (\tau)}.
\end{equation}
Let us fix $v^s\in S(\tau)$  and $v^u \in U(\tau)$ such that $\lVert v^s\rVert_\tau=\lVert v^u\rVert_\tau=1$. It follows from~\eqref{pb} that for all $t\ge \tau$,
\[
\lVert T(t, \tau)(v^s+v^u)\rVert_t \le M(t/\tau)^a \lVert v^s+v^u\rVert_\tau,
\]
and thus by~\eqref{201} and~\eqref{202} we have that 
\begin{align}\label{iu}
\begin{split}
\lVert v^s+v^u\rVert_\tau &\ge \frac{1}{M(t/\tau)^a} \lVert T(t, \tau)(v^s+v^u)\rVert_t \\
&\ge \frac{1}{M(t/\tau)^a} \bigg{(}\lVert T(t, \tau)v^u\rVert_t-\lVert T(t, \tau)v^s\rVert_t \bigg{)} \\
&\ge \frac{1}{M(t/\tau)^a} \bigg{(}\frac 1 D (t/\tau)^\lambda -D(t/\tau)^{-\lambda} \bigg{)}. \\
\end{split}
\end{align}
Choose now $N_0\in \N$ such that 
\[
\frac 1 DN_0^\lambda -DN_0^{-\lambda}>0. 
\]
Hence, it follows from~\eqref{iu} (by taking $t=N_0\tau$) that 
\[
\lVert v^s+v^u\rVert_\tau \ge \frac{1}{MN_0^a} \bigg{(}\frac 1 DN_0^\lambda -DN_0^{-\lambda} \bigg{)}=:c>0.
\]
Therefore, $\gamma (\tau)\ge c$ and thus the conclusion of the lemma follows readily from~\eqref{Pn}.
\end{proof}
The conclusion of the theorem now follows directly from~\eqref{201}, \eqref{202} and~\eqref{203}.
\end{proof}

We  now discuss Theorems~\ref{t2} and~\ref{t1} in the particular case of polynomial contractions and expansions.  We say that an  evolution family $T(t,\tau)$  admits a \emph{polynomial contraction} with respect to the family of norms $\lVert \cdot \rVert_t$ if it admits a polynomial dichotomy with respect to the family  of norms $\lVert \cdot \rVert_t$ and with
projections $P(t)=\Id$, $t\ge 1$.

Similarly, we say that an evolution family  $T(t,\tau)$ admits a  \emph{polynomial expansion} with respect to the family of norms $\lVert \cdot \rVert_t$ if it admits a polynomial dichotomy with respect to the family of norms $\lVert \cdot \rVert_t$ and with
projections $P(t)=0$, $t\ge 1$.

The following two results are essentially  direct consequences of Theorems~\ref{t2} and~\ref{t1}.
\begin{theorem}
Assume that an evolution family $T(t,\tau)$ satisfies~\eqref{pb}  with $M,a>0$. 
The following two statements are equivalent:
\begin{itemize}
\item $T(t,\tau)$ admits a polynomial contraction with respect to the family of norms $\lVert \cdot \rVert_t$;
\item for each $y\in Y_1$, $x\colon [1, \infty) \to X$ defined by
\begin{equation}\label{xf}
x(t)=\int_1^t \frac 1 s T(t,s)y(s)\, ds \quad t\ge 1, 
\end{equation}
belongs to $Y_0$. 
\end{itemize}
\end{theorem}

\begin{proof}
By proceeding as in the proof of Theorem~\ref{t2}, it is easy to show that the first statement implies the second.  Conversely, under the assumption that the second statement is valid, we have that that the assumptions of Theorem~\ref{t1} are valid with $Z=\{0\}$ and thus the desired conclusion follows. 
\end{proof}

\begin{theorem}
Assume that an evolution family $T(t,\tau)$ satisfies~\eqref{pb}  with $M,a>0$. 
The following two statements are equivalent:
\begin{itemize}
\item $T(t,\tau)$ admits a polynomial expansion with respect to the family of norms $\lVert \cdot \rVert_t$;
\item for each $y\in Y_1$ there exists a unique $x\in Y_X$ satisfying~\eqref{503a}. 
\end{itemize}
\end{theorem}

\begin{proof}
The conclusion of the theorem follows directly from Theorems~\ref{t2} and~\ref{t1}.
\end{proof}

We stress that it was proved in~\cite{BDV3} that the version of Theorem~\ref{t1} for classical exponential dichotomies holds without an assumption of the type~\eqref{pb}. Therefore, it is natural to ask if the conclusion of Theorem~\ref{t1} is valid in the absence of~\eqref{pb}. 
The following example shows that the answer to this question is negative. 
\begin{example}
Let $X=\mathbb R$ with the standard Euclidean norm $\lvert \cdot \rvert$. Furthermore, let $\lVert \cdot \rVert_t=\lvert \cdot \rvert$ for $t\ge 1$. We consider the sequence $(A_n)_{n\in \N}$ of operators (which can be identified with numbers) on $X$  given by
\[
A_n=\begin{cases}
n & \text{if $n=2^l$ for some $l\in \N$;}\\
0 & \text{otherwise.}
\end{cases}
\]
Furthermore, for $t\ge \tau \ge 1$ we define
\[
T(t, \tau)=\begin{cases}
A_{\lfloor t\rfloor-1} \cdots A_{\lfloor \tau \rfloor} & \text{if $\lfloor t\rfloor \ge \lfloor \tau \rfloor +1$;}\\
\Id & \text{if $\lfloor t\rfloor=\lfloor \tau \rfloor$.}
\end{cases}
\]
Clearly, $T(t, \tau)$ is an evolution family. It is easy to verify that for $y\in Y_1$, $x$ given by~\eqref{xf} satisfies
\[
x(t)=\begin{cases}
\int_{\lfloor t\rfloor}^t \frac{1}{s}y(s)\, ds +\int_{\lfloor t \rfloor -1}^{\lfloor t\rfloor}\frac{1}{s}A_{\lfloor t\rfloor -1}y(s)\, ds & \text{if $\lfloor t\rfloor=2^l+1$ for $l\in \N$;}\\
\int_{\lfloor t\rfloor}^t \frac{1}{s}y(s)\, ds & \text{otherwise.}
\end{cases}
\]
Hence, $x\in Y_0$ and in fact $\lVert x\rVert_\infty \le 2\lVert y\rVert_L$. However, $T(t, \tau)$ obviously doesn't admits a polynomial contraction since $\sup_{n\in \N}\lVert T(n+1, n)\rVert=\infty$.

\end{example}

\section{Nonuniform polynomial dichotomies}\label{NPD}
In this section we  recall the notion of a nonuniform exponential dichotomy and establish its connection with the notion of a polynomial dichotomy with respect to a family  of norms. 

We say that an evolution family  $T(t, \tau)$ admits a \emph{nonuniform polynomial dichotomy} if:
\begin{itemize}
\item there exist projections $P(t)$, $t\ge 1$ satisfying~\eqref{pro} and 
such that the map  $T(t,\tau)\rvert_{\Ker P(\tau)} \colon \Ker P(\tau)  \to \Ker P(t)$ is invertible for all $t\ge \tau \ge 1$;
\item there exist $\lambda, D>0$ and $\epsilon \ge 0$ such that for  $t, \tau\ge 1$ we have 
\begin{equation}\label{nd1}
\lVert T(t, \tau)P(\tau) \rVert \le D(t/\tau)^{-\lambda}\tau^\epsilon  \quad \text{for $t\ge \tau$}
\end{equation}
and 
\begin{equation}\label{nd2}
\lVert T(t,\tau)Q(\tau) \rVert \le  D(\tau/t)^{-\lambda}\tau^\epsilon  \quad \text{for $t\le \tau$,}
\end{equation}
where $Q(t)=\Id-P(t)$ and
\[
T(t,\tau)=(T(\tau, t)\rvert_{\Ker P(t)})^{-1} \colon \Ker P(\tau) \to \Ker P(t)
\]
for $t<\tau$.
\end{itemize}

\begin{proposition}\label{prop1}
The following properties are equivalent:
\begin{enumerate}
\item $T(t,\tau)$ admits a nonuniform polynomial dichotomy;
\item $T(t,\tau)$ admits a polynomial dichotomy with respect to a family  of norms $\lVert \cdot \rVert_t$ satisfying
\begin{equation}\label{ln}
\lVert x\rVert \le \lVert x\rVert_t \le Ct^\epsilon \lVert x\rVert \quad x\in X, \  t\ge 1
\end{equation}
for some $C>0$ and $\epsilon \ge 0$.
\end{enumerate}
\end{proposition}

\begin{proof}
Assume first that  $T(t, \tau)$ admits a nonuniform polynomial dichotomy. For each $\tau \ge 1$ and $x\in X$, let
\[
\lVert x\rVert_\tau:=\sup_{t\ge \tau} (\lVert  T(t,\tau)P(\tau) x\rVert (t/\tau)^\lambda)+\sup_{t\le \tau}(\lVert T(t, \tau)Q(\tau) x\rVert (\tau/t)^\lambda).
\]
It follows readily from~\eqref{nd1} and~\eqref{nd2} that~\eqref{ln} holds with $C=2D$. Furthermore, for $t\ge \tau$ and $x\in X$ we have that 
\[
\begin{split}
\lVert T(t, \tau)P(\tau) x\rVert_t &=\sup_{s\ge t} (\lVert T(s, t)T(t, \tau)P(\tau) x\rVert  (s/t)^\lambda) \\
&\le \sup_{s\ge \tau} (\lVert T(s, \tau)P(\tau) x\rVert  (s/t)^\lambda) \\
&=(t/\tau)^{-\lambda} \sup_{s\ge \tau}(\lVert T(s, \tau)P(\tau) x\rVert  (s/\tau)^\lambda) \\
&=(t/\tau)^{-\lambda} \lVert x\rVert_\tau,
\end{split}
\]
and thus~\eqref{d1} holds. Similarly, one can show that~\eqref{d2} holds. Therefore, $T(t,\tau)$ admits a polynomial dichotomy with respect to the family of norms $\lVert \cdot \rVert_t$.

Conversely, suppose that $T(t,\tau)$ admits a polynomial dichotomy with respect to a family of norms $\lVert \cdot \rVert_t$ satisfying~\eqref{ln} for some $C>0$ and $\epsilon \ge 0$. It follows that~\eqref{d1} and~\eqref{ln} that
\[
\begin{split}
\lVert T(t, \tau)P(\tau) x\rVert  &\le \lVert T(t, \tau)P(\tau) x\rVert_t  \\
& \le D(t/\tau)^{-\lambda}\lVert x\rVert_\tau \\
&\le CD(t/\tau)^{-\lambda}\tau^\epsilon \lVert x\rVert,
\end{split}
\]
for $t\ge \tau$ and $x\in X$. Therefore, \eqref{nd1} holds. Similarly, one can establish~\eqref{nd2} and therefore $T(t,\tau)$ admits a nonuniform polynomial dichotomy.
\end{proof}
However,  the norms $\lVert \cdot \rVert_t$ constructed in the proof of Proposition~\ref{prop1} can fail to satisfy~\eqref{pb}. 
Therefore, in order to be able to apply our main results, we will consider a stronger notion of a nonuniform polynomial dichotomy. 

We say that $T(t,\tau)$ admits a \emph{strong nonuniform polynomial dichtotomy} if it admits a nonuniform polynomial dichotomy and there exist $K, b>0$ such that
\[
\lVert T(t, \tau)\rVert \le K (t/\tau)^b n^\epsilon \quad \text{for $t\ge \tau$.}
\]

\begin{proposition}\label{12}
The following properties are equivalent:
\begin{enumerate}
\item $T(t,\tau)$ admits a strong  nonuniform polynomial dichotomy;
\item $T(t,\tau)$ admits a polynomial dichotomy with respect to a family of norms $\lVert \cdot \rVert_t$ satisfying~\eqref{pb} and~\eqref{ln}
for some $C, M, a>0$ and $\epsilon \ge 0$.
\end{enumerate}
\end{proposition}

\begin{proof}
Assume that  $T(t,\tau)$ admits a strong  nonuniform polynomial dichotomy. For $\tau \ge 1$ and $x\in X$, set 
\[
\begin{split}
\lVert x\rVert_\tau &=\sup_{t\ge \tau} (\lVert T(t, \tau)P(\tau) x\rVert (t/\tau)^\lambda)+\sup_{t\le \tau}(\lVert T(t, \tau)Q(\tau) x\rVert (\tau/t)^\lambda)\\
&\phantom{=}+\sup_{t\ge \tau}(\lVert T(t, \tau)Q(\tau) x\rVert (t/\tau)^{-b}).
\end{split}
\]
By repeating the arguments in the proof of Proposition~\ref{prop1}, it is easy to verify that  $T(t,\tau)$ admits a polynomial dichotomy with respect to the family of norms $\lVert \cdot \rVert_t$ and that~\eqref{pb} and~\eqref{ln} holds. The converse can also be obtained by arguing as in the proof of Proposition~\ref{prop1}.
\end{proof}

\section{Robustness of strong nonuniform polynomial dichotomies}\label{R}
In this section we apply our main results to establish to prove that the notion of a strong nonuniform polynomial dichotomy persists under sufficiently small linear perturbations.

\begin{theorem}\label{qe}
Assume that the evolution family $T(t,\tau)$ admits a strong nonuniform polynomial dichotomy and that $B \colon [1, \infty) \to B(X)$ is a strongly continuous function such that 
\begin{equation}\label{robustness}
\lVert B(t) \rVert \le \frac{c}{t^{1+\epsilon}} \quad \text{for $t\ge 1$.}
\end{equation}
For any sufficiently small $c>0$, the evolution family $U(t,\tau)$ satisfying 
\[
U(t,\tau)=T(t,\tau)+\int_\tau^t T(t,s)B(s)U(s,\tau)\, ds
\]
admits a strong nonuniform polynomial dichotomy.

\end{theorem}

\begin{proof}
Since $T(t,\tau)$ admits a strong nonuniform polynomial dichotomy, it follows from Proposition~\ref{12} that there exists a family of norms $\lVert \cdot \rVert_t$, $t\ge 1$ such that~\eqref{pb} and~\eqref{ln} hold for some $C, M, a >0$ and with $\epsilon \ge 0$ as in the definition of the notion of a (strong) nounuiform exponential dichotomy. Moreover,
$T(t,\tau)$ admits a polynomial dichotomy with respect to the  family of norms $\lVert \cdot \rVert_t$. Hence, Theorem~\ref{t2} implies that there exists a closed subspace $Z\subset X$ such that the operator $T_Z \colon \mathcal D(T_Z)\to Y_1$ (defined in the proof  Theorem~\ref{t1}) is invertible.  Furthermore, Lemma~\ref{9:26} implies that $T_Z$ is closed. 
For $x\in \mathcal D(T_Z)$, we consider the graph norm $\lVert x\rVert_{T_Z}:=\lVert x\rVert_\infty+\lVert T_Zx\rVert_1$. Since $T_Z$ is closed, $(\mathcal D(T_Z), \lVert \cdot \rVert_{T_Z})$ is a Banach space. Moreover, the operator $T_Z\colon \mathcal D(T_Z)\to Y_1$  is bounded and from now on we denote it simply by $T_Z$.

We define $D\colon \mathcal D(T_Z)\to Y_1$ by
\[
(Dx)(t)=tB(t)x(t) \quad \text{for $t\ge 1$ and $x\in \mathcal D(T_Z)$.}
\]
It follows from~\eqref{ln} and~\eqref{robustness} that
\[
\begin{split}
\lVert Dx\rVert_L &=\sup_{t\ge 1}\int_t^{t+1}\lVert (Dx)(s)\rVert_s \, ds \\
&=\sup_{t\ge 1}\int_t^{t+1}s \lVert B(s) x(s)\rVert_s \, ds \\
&\le C\sup_{t\ge 1}\int_t^{t+1}s^{1+\epsilon} \lVert B(s) x(s)\rVert \, ds \\
&\le cC\sup_{t\ge 1}\int_t^{t+1}\lVert x(s)\rVert \, ds \\
&\le  cC\sup_{t\ge 1}\int_t^{t+1}\lVert x(s)\rVert_s \, ds \\
&\le cC\lVert x\rVert_\infty, 
\end{split}
\]
and thus
\begin{equation}\label{est}
\lVert Dx\rVert_L \le cC \lVert x\rVert_{T_Z} \quad \text{for $x\in \mathcal D(T_Z)$.}
\end{equation}

Moreover, we define a linear operator $U_Z \colon \mathcal D(U_Z)\to Y_1$ defined by $U_Z x=y$ on the domain $\mathcal D(U_Z)$ formed by all $x\in Y_Z$ for which there exists $y\in Y_1$ such that
\[
x(t)=U(t,\tau)x(\tau)+\int_\tau^t \frac 1 s U(t,s)y(s)\, ds \quad \text{for $t\ge \tau$.}
\]
\begin{lemma}
We have that
\begin{equation}\label{6:00}
\mathcal D(T_Z)=\mathcal D(U_Z)  \quad \text{and} \quad T_Z=U_Z+D.
\end{equation}
\end{lemma}

\begin{proof}[Proof of the lemma]
Take $x\in Y_Z$ and $y\in Y_1$ such that $U_Z x=y$. Then,
\[
\begin{split}
x(t)&=U(t,\tau)x(\tau)+\int_\tau^t \frac 1 s U(t,s)y(s)\, ds  \\
&=T(t,\tau)x(\tau)+\int_\tau^tT(t,s)B(s)U(s,\tau)x(\tau)\, ds \\
&\phantom{=}+\int_\tau^t \frac 1 s T(t,s)y(s)\, ds+\int_\tau^t \int_s^t \frac 1 s T(t, w)B(w)U(w,s)y(s)\, dw \, ds \\
&=T(t,\tau)x(\tau)+\int_\tau^tT(t,w)B(w)U(w,\tau)x(\tau)\, dw \\
&\phantom{=}+\int_\tau^t \frac 1 s T(t,s)y(s)\, ds+\int_\tau^t \int_\tau^w \frac 1 s T(t, w)B(w)U(w,s)y(s)\, ds \, dw \\
&=T(t,\tau)x(\tau)+\int_\tau^t \frac 1 s T(t,s)y(s)\, ds\\
&\phantom{=}+\int_\tau^t T(t, w)B(w)\bigg{(}U(w,\tau)x(\tau)+\int_\tau^w \frac 1 s U(\omega, s)y(s)\, ds \bigg{)}\, dw \\
&=T(t,\tau)x(\tau)+\int_\tau^t \frac{1}{w}T(t,w)\bigg{(}y(w)+wB(w)x(w)\bigg{)}\, dw,
\end{split}
\]
for $t\ge \tau$. Therefore, $x\in \mathcal D(T_Z)$ and $T_Zx=U_Zx+Dx$. Reversing the arguments, we also obtain that $\mathcal D(T_Z)\subset \mathcal D(U_Z)$ and the proof of the lemma is completed. 
\end{proof}
It follows from~\eqref{est} and~\eqref{6:00} that 
\begin{equation}\label{6:05}
\lVert (U_Z-T_Z)x\rVert_L=\lVert Dx\rVert_L \le cC\lVert x\rVert_\infty \le  cC\lVert x\rVert_{T_Z}
\end{equation}
for $x\in \mathcal D(T_Z)$. Hence, the linear operator
\[
U_Z \colon (\mathcal D(T_Z), \lVert \cdot \rVert_{T_Z}) \to Y_1
\]
is bounded. Moreover, it follows from~\eqref{6:05} and  the invertibility of $T_Z$ that if $c$ is sufficiently small, then $U_Z$ is also invertible.  

On the other hand,  it follows from~\eqref{pb}, \eqref{ln} and~\eqref{robustness} that
\[
\begin{split}
\lVert U(t,\tau)x\rVert_t &=\bigg{\lVert} T(t, \tau)x+\int_\tau^t T(t,s)B(s)U(s, \tau)x\, ds \bigg{\rVert}_t  \\
&\le M(t/\tau)^a \lVert x\rVert_\tau+\int_\tau^t cCM(t/s)^a \frac 1 s  \lVert U(s, \tau)x\rVert_s\, ds 
\end{split}
\]
for $t\ge \tau$. Hence, the function $\phi(t)=t^{-a}\lVert U(t,\tau)x\rVert_t$ satisfies
\[
\phi(t)\le M\phi(\tau)+cCM\int_\tau^t \frac 1 s\phi(s)\, ds.
\]
Therefore, it follows from Gronwall lemma that 
\[
\phi(t)\le M\phi(\tau)(t/\tau)^{cCM},
\]
and thus
\begin{equation}\label{845}
\lVert U(t,\tau)x\rVert_t \le M(t/\tau)^{a+cCM} \lVert x\rVert_\tau \quad \text{for $t\ge \tau$ and $x\in X$.}
\end{equation}
Since $U_Z$ is invertible and~\eqref{845} holds, it follows from Theorem~\ref{t1} that $U(t,\tau)$ admits a polynomial dichotomy with respect to the  family of norms $\lVert \cdot\rVert_t$. It then follows from Proposition~\ref{12}
that $U(t,\tau)$ admits a nonuniform polynomial dichotomy and the proof of the theorem is completed.

\end{proof}

\begin{remark}
Although the robustness property of nonuniform polynomial dichotomy follows from more general results established in~\cite{BV2}, the approach we give is new and can be seen as a natural extension of the treatment of the robustness property for exponential dichotomies. 
Besides this, it should be noted that our condition for robustness (see~\eqref{robustness}) is weaker than the one required in~\cite[Theorem 1.]{BV2} (in the particular case of polynomial dichotomies).
\end{remark}
\bibliographystyle{amsplain}

\end{document}